\documentclass[12pt]{amsproc}%
\usepackage{amssymb}
\usepackage{amsmath}
\usepackage{amsfonts}
\usepackage{geometry}
\usepackage{graphicx}%
\providecommand{\U}[1]{\protect\rule{.1in}{.1in}}
\theoremstyle{plain}

\newtheorem{definition}{Definition}

\newtheorem{lemma}{Lemma}

\newtheorem{theorem}{Theorem}
\numberwithin{equation}{section}
\begin{document}
\title[Subcritical and critical growth without the AR condition]{N-Laplacian equations in $%
\mathbb{R}
^{N}$ with subcritical and critical growth without the Ambrosetti-Rabinowitz condition.}
\author{Nguyen Lam}
\author{Guozhen Lu}
\address{Nguyen Lam and Guozhen Lu\\
Department of Mathematics\\
Wayne State University\\
Detroit, MI 48202, USA\\
Emails: nguyenlam@wayne.edu and gzlu@math.wayne.edu}
\thanks{Corresponding Author: G. Lu at gzlu@math.wayne.edu}
\thanks{Research is partly supported by a US NSF grant \#DMS0901761.}
\date{}
\dedicatory{ }
\begin{abstract}
Let $\Omega$ be a bounded domain in $%
\mathbb{R}
^{N}$. In this paper, we consider the following nonlinear elliptic equation of
$N$-Laplacian type:
\begin{equation}
\left\{
\begin{array}
[c]{l}%
-\Delta_{N}u=f\left(  x,u\right) \\
u\in W_{0}^{1,2}\left(  \Omega\right)  \setminus\left\{  0\right\}
\end{array}
\right.  \label{0.1}%
\end{equation}
when $f$ is of subcritical or critical exponential growth. This nonlinearity
is motivated by the Moser-Trudinger inequality. In fact, we will prove the
existence of a nontrivial nonnegative solution to (\ref{0.1}) without the
Ambrosetti-Rabinowitz $(AR)$ condition. Earlier works in the literature on the
existence of nontrivial solutions to $N-$Laplacian in $\mathbb{R}^{N}$ when
the nonlinear term $f$ has the exponential growth only deal with the case when
$f$ satisfies the $(AR)$ condition. Our approach is based on a suitable
version of the Mountain Pass Theorem introduced by G. Cerami \cite{Ce1, Ce2}.
This approach can also be used to yield an existence result for the
$p$-Laplacian equation ($1<p<N$) in the subcritical polynomial growth case.

\end{abstract}
\maketitle

\section{Introduction}

Let $\Omega$ be a bounded smooth domain in $%
\mathbb{R}
^{N}$ and we consider the following class of nonlinear elliptic equations
\begin{equation}
\left\{
\begin{array}
[c]{l}%
-\Delta_{p}u=f\left(  x,u\right)  \text{ in }\Omega,\\
u\in W_{0}^{1,p}\left(  \Omega\right)  \setminus\left\{  0\right\}
\end{array}
\right.  \label{1.1}%
\end{equation}
where $-\Delta_{p}u=-div\left(  |\nabla u|^{p-2}\nabla u\right)  $ is the
$p-$Laplacian. It is well known that problems involving the $p-$Laplacian
appear in many contexts. Some of these problems come from different areas of
applied mathematics and physics. For example, they may be found in the study
of non-Newtonian fluids, nonlinear elasticity and reaction-diffusions. 
The main purpose of this paper is to establish existence results of nontrivial nonnegative solutions to the above problem of $N-$Laplacian when the nonlinear term $f$ has the exponential growth but without satisfying the Ambrosetti-Rabinowitz condition.
In these cases, the original version of the Mountain Pass Theorem of Ambrosetti-Rabinowitz \cite{AR, R} is not sufficient for our purpose. Therefore, we will adapt a suitable version of
Mountain Pass Theorem introduced by Cerami \cite{Ce1, Ce2} to accomplish our goal. Our approach also yields an existence result of nontrivial nonnegative solutions when $1<p<N$ and $f$ satisfies a certain subcritical polynomial growth condition weaker than those in the literature.

In the case $p=N$, motivated by the Trudinger-Moser inequality (see Lemma 3),
existence of nontrivial solutions to $N-$Laplacian when $f$ has the
exponential growth have been studied by many authors. See for example,
Carleson-Chang \cite{CC}, Atkinson-Peletier \cite{AP2}, Adimurthi et al
\cite{A, AP1, AS1, AS2, AY1, AY2}, Marcos Do O et al \cite{Mdo, MdoS, O, OEU},
de Figueiredo et al \cite{FDR, FMR}, etc. using the classical Critical Point
Theory first developed by Ambrosetti-Rabinowitz in their celebrated work
\cite{AR}, see also \cite{R}. The key issue in using such a theory is the
verification of conditions which allow the use of the Palais-Smale condition.

When $1<p<N$, there have been substantial amount of works to study the
existence of the nontrivial solution for (\ref{1.1}). Nevertheless, almost all
of the works involve the nonlinear term $f(x,u)$ of a subcritical (polynomial)
growth, say,

$(SCP):\,\,~$There exist positive constants $c_{1}$ and $c_{2}$ and $q_{0}%
\in\left(  p-1,p^{\ast}-1\right)  $ such that
\[
0\leq f(x,t)\leq c_{1}+c_{2}t^{q_{0}}\text{ for all }t\geq0\text{ and }%
x\in\overline{\Omega}%
\]
where $p^{\ast}=Np/(N-p)$ denotes the critical Sobolev exponent. In this case,
we can treat the problem (\ref{1.1}) variationally in the Sobolev space
$W_{0}^{1,p}\left(  \Omega\right)  $ thanks to the standard Mountain Pass
Theorem. Since Ambrosetti and Rabinowitz proposed the Mountain-pass Theorem in
their celebrated paper \cite{AR}, critical point theory has become one of the
main tools for finding solutions to elliptic equations of variational type.
Indeed, if we define the Euler-Lagrange function associated to problem
(\ref{1.1}):%

\begin{align*}
J  &  :W_{0}^{1,p}\left(  \Omega\right)  \rightarrow%
\mathbb{R}%
\\
J(u)  &  =\frac{1}{p}\int_{\Omega}\left\vert \nabla u\right\vert ^{p}%
dx-\int_{\Omega}F(x,u)dx
\end{align*}
where%
\[
F(x,u)=%
{\displaystyle\int\limits_{0}^{u}}
f(x,s)ds
\]
then the critical point of $J$ are precisely the weak solutions of problem
(\ref{1.1}). One of the main conditions that appeared in many works is the
so-called Ambrosetti-Rabinowitz condition:

$(AR):\,\,$ There are constants $\theta>p$ and $s_{0}>0$ such that%
\[
0<\theta F(x,s)\leq sf(x,s),~\left\vert s\right\vert \geq s_{0},~\forall
x\in\Omega
\]
In fact, the $(AR)$ condition is quite natural and plays an important role in
studying problem (\ref{1.1}), for example, it ensures the boundedness of the
Palais-Smale sequence. On the other hand, this condition is very restrictive
and eliminates many interesting and important nonlinearities. We recall that
$(AR)$ condition implies another weaker condition%

\[
f\text{ is p-superlinear at infinity, i.e., }\underset{n\rightarrow\infty
}{\lim}\frac{f(x,t)}{\left\vert t\right\vert ^{p-1}}=+\infty,\text{ uniformly
in }x\in\Omega.
\]
However, there are many functions which satisfy the p-superlinearity at
infinity, but do not satisfy the $(AR)$ condition. An example of such
functions is
\[
f(x,t)=\left\vert t\right\vert ^{p-2}t\log(1+\left\vert t\right\vert ).
\]

Over the years, many researchers studied problem (\ref{1.1}) by trying to drop
the $(AR)$ condition, see for instance \cite{ILU, J, LL, LW, LWZ, LY, LZ, MS,
PR, S, SZ, WZ, Z}. For example, the following assumption has been studied by
many authors:%

\[
\frac{f(x,t)}{\left\vert t\right\vert ^{p-1}}\text{ is non-decreasing with
respect to }\left\vert t\right\vert
\]
(see \cite{LZ, MS, SZ} and references therein). Recently, the authors of
\cite{FL} have used the following condition:%
\[
\text{There exists }\theta\geq1\text{ such that }\theta G(x,t)\geq
G(x,st)\text{ for all }\left(  x,t\right)  \in\Omega\times%
\mathbb{R}
\text{ and }s\in\left[  0,1\right]
\]
where $G(x,t)=f(x,t)t-pF(x,t)$, to compute the critical groups of the
functional $J$ at infinity, and obtain one nontrivial solution of (\ref{1.1}).
This condition was first introduced by Jeanjean \cite{J}, and then was used by
numerous authors, for example, \cite{LW, LY, MS, S, WZ}.

We note  that except in \cite{LW}, the other authors assumed the condition
$(SCP)$ in their works in order to get the existence results. One of the main
reasons to assume this condition (SCP) is that they can use the Sobolev
compact embedding $W_{0}^{1,p}\left(  \Omega\right)  \hookrightarrow
L^{q}\left(  \Omega\right)  $, $1\leq q<p^{\ast}$.

In this paper, our first main result will be to study problem (\ref{1.1}) in
the improved subcritical polynomial growth
\[
(SCPI): \,\, \underset{s\rightarrow+\infty}{\lim}\frac{f\left(  x,s\right)
}{\left\vert s\right\vert ^{p^{\ast}-1}}=0
\]
which is much weaker than $(SCP)$. Note that in this case, we don't have the
Sobolev compact embedding anymore. Our work again is without the
$(AR)-$condition. In fact, this condition was studied by Liu and Wang in \cite{LW} in the case
of Laplacian (i.e., $p=2$) by the Nehari manifold approach. However, we will
show that we can use a suitable version of the Mountain Pass Theorem to get
the nontrivial solution to (\ref{1.1}) in the general case $1<p<N$. This
result is stronger than those in \cite{ILU, LY, MS, S}.

Let us now state our result: Consider the problem:
\begin{equation}
\left\{
\begin{array}
[c]{l}%
-\Delta_{p}u=f\left(  x,u\right)  \text{ in }\Omega,\\
u\in W_{0}^{1,p}\left(  \Omega\right)  \setminus\left\{  0\right\} \\
u\geq0
\end{array}
\right.  \tag{P}\label{1.2}%
\end{equation}

Suppose that

$(L1):\,\, ~f:\Omega\times%
\mathbb{R}
\rightarrow%
\mathbb{R}
$ is continuous, $f\left(  x,u\right)  \geq0,~\forall\left(  x,u\right)
\in\Omega\times\left[  0,\infty\right)  $ and $f\left(  x,u\right)
=0,~\forall\left(  x,u\right)  \in\Omega\times\left(  -\infty,0\right]  $.

$(L2):\,\, ~\underset{u\rightarrow+\infty}{\lim}\frac{F\left(  x,u\right)
}{u^{p}}=+\infty$ uniformly on $x\in\Omega$ where $F(x,u)=\int\limits_{0}%
^{u}f(x,t)dt.$

$(L3):\,\, $ There is $C_{\ast}\geq0,~\theta\geq1$ such that $H(x,t)\leq\theta
H(x,s)+C_{\ast}$ for all $0<t<s,~\forall x\in\Omega$ where
$H(x,u)=uf(x,u)-pF(x,u).$

$(L4):\,\,$ $\underset{u\rightarrow0+}{\lim\sup}\frac{pF\left(  x,u\right)
}{\left\vert u\right\vert ^{p}}<\lambda_{1}\left(  \Omega\right)  $, uniformly
on $x\in\Omega.$

where
\[
\lambda_{1}\left(  \Omega\right)  =\inf\left\{  \frac{\int\left\vert \nabla
u\right\vert ^{p}dx}{\int\left\vert u\right\vert ^{p}dx}:~u\in W_{0}%
^{1,p}\left(  \Omega\right)  \setminus\left\{  0\right\}  \right\}
\]
then

\begin{theorem}
Let $1<p<N$ and assume that $f$ has the improved subcritical polynomial growth
on $\Omega$ (condition $(SCPI)$) and satisfies $(L1),~(L2),~(L3)$ and $(L4)$.
Then, problem (\ref{1.2}) has a nontrivial solution.
\end{theorem}

Since we are only concerned with the nonnegative solution, the condition
$(L1)$ is natural. Moreover, condition $(L2)$ is just a consequence of the
p-superlinear at infinity of $f$. The type of condition $(L3)$ was first
introduced by Jeanjean \cite{J} and was used in subsequent works, see
\cite{ILU, LY, MS, S}. Finally, in earlier works (see e.g., \cite{ILU, LY,
MS}), they also often assumed that
\[
\underset{u\rightarrow0^{+}}{\lim}\frac{f\left(  x,u\right)  }{u^{p-1}%
}=0\text{ uniformly on }x\in\Omega
\]
which is stronger than our condition $(L4)$.

In case of $p=N$, we have $p^{\ast}=+\infty$. In this case, every polynomial
growth is admitted, but one knows by easy examples that $W_{0}^{1,N}\left(
\Omega\right)  \nsubseteq L^{\infty}\left(  \Omega\right)  $. Hence, one is
led to look for a function $g(s):%
\mathbb{R}
\rightarrow%
\mathbb{R}
^{+}$ with maximal growth such that
\[
\underset{u\in W_{0}^{1,N}\left(  \Omega\right)  ,~\left\Vert u\right\Vert
\leq1}{\sup}\int_{\Omega}g\left(  u\right)  dx<\infty
\]
It was shown by Trudinger \cite{T} and Moser \cite{M} that the maximal growth
is of exponential type. So, we must redefine the subcritical (exponential)
growth and the critical (exponential) growth in this case as follows:

$(SCE):\,\, $ $f$ has subcritical (exponential) growth on $\Omega
,i.e,\underset{u\rightarrow+\infty~}{\lim}\frac{\left\vert f\left(
x,u\right)  \right\vert }{\exp\left(  \alpha\left\vert u\right\vert
^{N/(N-1)}\right)  }=0,$ uniformly on $x\in\Omega\ $\ for all $\alpha
>0.\bigskip$

$(CG):\,\, $ $f$ has critical growth on $\Omega$, i.e., there exists
$\alpha_{0}>0$ such that
\[
\underset{u\rightarrow+\infty~}{\lim}\frac{\left\vert f\left(  x,u\right)
\right\vert }{\exp\left(  \alpha\left\vert u\right\vert ^{N/(N-1)}\right)
}=0,\text{ uniformly on }x\in\Omega,~\forall\alpha>\alpha_{0}%
\]
and
\[
\underset{u\rightarrow+\infty~}{\lim}\frac{\left\vert f\left(  x,u\right)
\right\vert }{\exp\left(  \alpha\left\vert u\right\vert ^{N/(N-1)}\right)
}=+\infty,\text{ uniformly on }x\in\Omega,~\forall\alpha<\alpha_{0}%
\]

When $p=N$ and $f$ has the subcritical exponential growth (SCE), again we can
use the Mountain Pass theorem together with the $(AR)$ condition to get the
nontrivial solution to (\ref{1.1}). Nevertheless, it seems that there are no
works when the nonlinear term $f$ does not satisfy the (AR) condition in this
case. Thus, the second main result of this paper is to establish the existence
of nontrivial nonnegative solutions to (\ref{1.2}) when $f$ has the
subcritical exponential growth (SCE). More precisely, we will study the
existence of the nonnegative nontrivial solution to problem (\ref{1.2}) where
we don't need to use the $(AR)$ condition.  Our result is as follows:

\begin{theorem}
Let $p=N$ and assume that $f$ has the subcritical exponential growth on
$\Omega$ (condition $(SCE)$) and satisfies $(L1),~(L2),~(L3)$ and $(L4)$.
Then, problem (\ref{1.2}) has a nontrivial solution.
\end{theorem}

When $p=N$ and $\ f$ has the critical exponential growth (CG), the study of
the problem (\ref{1.1}) becomes much more difficult than in the case of
subcritical exponential growth. Similar to the case of the critical polynomial
growth in $R^{N}$ ($N\geq3$) for the Laplacian studied by Brezis and Nirenberg
in their pioneering work \cite{BN}), our Euler-Lagrange functional does not
satisfy the Palais-Smale condition at all level anymore. Instead, the authors
in \cite{A, O, OEU} used the extremal function sequences related to
Moser-Trudinger inequality to prove that $J$ satisfies the Palais-Smale at a
certain level. Moreover, this Palais-Smale sequence was shown to be bounded
and then derived a nontrivial solution. The idea of choosing the testing
functions which are extremal to the Moser-Trudinger inequality is inspired by
the work of Brezis and Nirenberg where the testing functions are extremal to
the Sobolev embedding inequality.

However, in the works \cite{A, O, OEU}, they need to assume a much more
restrictive condition$~$%
\[
(ARR):\,\, ~\exists t_{0}>0,~\exists M>0\text{ such that }\forall\left\vert
u\right\vert \geq t_{0},~\forall x\in\Omega,0<F(x,u)\leq M\left\vert f\left(
x,u\right)  \right\vert
\]
It's clear that the condition (ARR) implies the $(AR)$ condition.

Our third main purpose of this paper is to study problem (\ref{1.2}) without
using the $(ARR)$ condition or $(AR)$ condition. Indeed, we get the following result:

\begin{theorem}
Let $p=N$ and assume (L1), (L2), (L3) with $\theta=1$ and $C^{\ast}=0$, (L4)
and that $f$ has critical growth on $\Omega$ $(CG)$, say, at $\alpha_{0}$.
Furthermore assume that

(L5):\,\, $\underset{t\rightarrow+\infty}{\lim}f\left(  x,t\right)
\exp\left(  -\alpha_{0}\left\vert t\right\vert ^{N/(N-1)}\right)  t\geq
\beta>\left(  \frac{N}{d}\right)  ^{N}\frac{1}{\mathcal{M}\alpha_{0}^{N-1}}$,
uniformly in $(x,t)$ where $d$ is the inner radius of $\Omega$, i.e. $d:=$
radius of the largest open ball $\subset\Omega;$
\[
\mathcal{M}=\underset{n\rightarrow\infty}{\lim}n%
{\displaystyle\int\limits_{0}^{1}}
\exp n\left(  t^{N/(N-1)}-t\right)  dt~(\geq2)
\]
and

(L6):\,\, $f$ is in the class $(L_{0}),$ i.e., for any $\left\{
u_{n}\right\}  $ in $W_{0}^{1,N}\left(  \Omega\right)  ,$ if $\left\{
\begin{array}
[c]{c}%
u_{n}\rightharpoonup0\text{ in }W_{0}^{1,N}\left(  \Omega\right)  \text{ }\\
f(x,u_{n})\rightarrow0\text{ in }L^{1}\left(  \Omega\right)
\end{array}
\right.  $, then $F(x,u_{n})\rightarrow0$ in $L^{1}\left(  \Omega\right)  $
(up to a subsequence).

Then, problem (\ref{1.2}) has a nontrivial solution.
\end{theorem}

It is easy to see that condition $(L2)$ in Theorem 3 is just a consequence of the critical exponential growth condition $(CG)$ and therefore it is automatically satisfied.

The following remarks are in order. First of all, in dimension two we have recently established in \cite{LaLu}
the existence of nontrivial nonnegative solutions to the Laplacian equation (i.e., $p=2$) when the nonlinear
term $f$ has the subcritical or critical exponential growth of order $\exp(\alpha u^2)$ but without satisfying the
Ambrosetti-Rabinowitz condition. These results in dimension two in \cite{LaLu} extend those of \cite{FMR} to the case when $f$ does not have the $(AR)$ condition. Second,
there have been many works in the literature in which the
$(AR)$ condition was replaced by other alternative conditions when $f$ has the polynomial growth. Our results in this paper appear to be the first time in high dimension for $N-$Laplacian when $f$ has the subcritical or critical exponential growth and without $(AR)$ condition.  

As far as the case when the nonlinear term $f$ has the polynomial growth is concerned,   we recall that, in \cite{WZo},
Willem and Zou used
\begin{align*}
H(x,s)\text{ is increasing in }s\text{, }\forall x  &  \in\Omega;\text{
}sf(x,s)\geq0~\forall s\in%
\mathbb{R}
,\\
~sf(x,s)  &  \geq C_{0}\left\vert s\right\vert ^{\mu},~\forall\left\vert
s\right\vert \geq s_{0}>0,~\forall x\in\Omega
\end{align*}
where $\mu>2$ and $C_{0}>0$, instead of $(AR)$. It's clear that this condition
is much stronger than our conditions. Also, in \cite{CM}, the authors replaced
$(AR)$ condition by
\[
\underset{s\rightarrow\infty}{\lim\inf}\frac{H\left(  x,s\right)  }{\left\vert
s\right\vert ^{\mu}}\geq k>0,\text{ uniformly a.e. }x\in\Omega,
\]
where $\mu\geq\mu_{0}>0$. In \cite{ScZ}, Schechter and Zou assumed that
\[
H(x,s)\text{ is convex in }s,~\forall x\in\Omega
\]
or there are constants $C>0$, $\mu>2$ and $r\geq0,$ such that
\[
\mu F\left(  x,t\right)  -tf\left(  x,t\right)  \leq C\left(  1+t^{2}\right)
,~\left\vert t\right\vert \geq r.
\]
As remarked in \cite{MS}, the later condition is in fact equivalent to $(AR)$
and it's easy to see that the convexity on $H$ is much stronger than our
condition. Indeed, observe that function $H(x,s)$ is a "quasi-monotonic"
function, and also if $H$ is monotonic function in $s<0$ and $s>0$, or a
convex function in $%
\mathbb{R}
$, then it satisfies $(L3)$ with $\theta=1$.

The organization of the paper is as follows. In section 2, we collect some
known results of Mountain Pass Theorem in critical point theory (\cite{AR},
\cite{R}, \cite{Ce1}, \cite{Ce2}). In particular, it is necessary to adapt the
appropriate version of the Mountain Pass Theorem  due to Cerami \cite{Ce1, Ce2} to remove the
Ambrosetti-Rabinowitz condition. Section 3 provides the proof of Theorem 1,
i.e., the existence of nontrivial nonnegative solutions to Problem (\ref{1.2})
when the nonlinear term $f$ has the improved subcritical polynomial growth
(SCPI). Section 4 deals with the case when the nonlinear term $f$ has the
subcritical exponential growth and gives the proof of Theorem 2. Section 5
contains the proof of Theorem 3 and establishes the existence of nontrivial
solutions when $f$ has the critical exponential growth.

\section{Preliminaries and Mountain Pass Theorems}

Let $\Omega$ be a bounded domain in $%
\mathbb{R}
^{N}$. We denote
\begin{align*}
\left\Vert u\right\Vert  &  =\left(  \int_{\Omega}\left\vert \nabla
u\right\vert ^{p}dx\right)  ^{1/p}\\
\left\Vert u\right\Vert _{p}  &  =\left(  \int_{\Omega}\left\vert u\right\vert
^{p}dx\right)  ^{1/p}\\
\lambda_{1}\left(  \Omega\right)   &  =\inf\left\{  \frac{\left\Vert
u\right\Vert ^{p}}{\left\Vert u\right\Vert _{p}^{p}}:~u\in W_{0}^{1,p}\left(
\Omega\right)  \setminus\left\{  0\right\}  \right\} \\
d  &  =\text{radius of the largest open ball }\subset\Omega
\end{align*}
Define the Euler-Lagrange functional associated to problem (\ref{1.2}):
\[
J(u)=\frac{1}{p}\left\Vert u\right\Vert ^{p}-\int_{\Omega}F(x,u)dx,\text{
}u\in W_{0}^{1,p}\left(  \Omega\right)
\]
From the hypotheses on $f$, by the standard arguments and the Moser-Trudinger
inequality (see Lemma 3), we can easily see that $J$ is well-defined. Also,
it's standard to check that $J$ is $C^{1}\left(  W_{0}^{1,p}\left(
\Omega\right)  ,%
\mathbb{R}
\right)  $ and
\[
DJ(u)v=\int_{\Omega}\left\vert \nabla u\right\vert ^{p-2}\nabla u\nabla
vdx-\int_{\Omega}f(x,u)vdx,~v\in W_{0}^{1,p}\left(  \Omega\right)
\]
Thus, the critical point of $J$ are precisely the weak solutions of problem
(\ref{1.2}). We will prove the existence of such critical points by the
Mountain Pass Theorem.

\begin{definition}
Let $\left(  X,\left\Vert \cdot\right\Vert _{X}\right)  $ be a real Banach
space with its dual space $\left(  X^{\ast},\left\Vert \cdot\right\Vert
_{X^{\ast}}\right)  $ and $I\in C^{1}\left(  X,%
\mathbb{R}
\right)  $. For $c\in%
\mathbb{R}
$, we say that $I$ satisfies the $(PS)_{c}$ condition if for any sequence
$\left\{  x_{n}\right\}  \subset X$ with
\[
I\left(  x_{n}\right)  \rightarrow c,~DI\left(  x_{n}\right)  \rightarrow
0\text{ in }X^{\ast}%
\]
there is a subsequence $\left\{  x_{n_{k}}\right\}  $ such that $\left\{
x_{n_{k}}\right\}  $ converges strongly in $X$. Also, we say that $I$
satisfies the $(C)_{c}$ condition if for any sequence $\left\{  x_{n}\right\}
\subset X$ with
\[
I\left(  x_{n}\right)  \rightarrow c,~\left\Vert DI\left(  x_{n}\right)
\right\Vert _{X^{\ast}}\left(  1+\left\Vert x_{n}\right\Vert _{X}\right)
\rightarrow0
\]
there is a subsequence $\left\{  x_{n_{k}}\right\}  $ such that $\left\{
x_{n_{k}}\right\}  $ converges strongly in $X$.
\end{definition}

We have the following versions of the Mountain Pass Theorem (see \cite{AR,
Ce1, Ce2, LY}):

\begin{lemma}
Let $\left(  X,\left\Vert \cdot\right\Vert _{X}\right)  $ be a real Banach
space and $I\in C^{1}\left(  X,%
\mathbb{R}
\right)  $ satisfies the $(C)_{c}$ condition for any $c\in%
\mathbb{R}
$, $I(0)=0$ and

$(i)$ There are constants $\rho,~\alpha>0$ such that $I|_{\partial B_{\rho}%
}\geq\alpha$.

$(ii)$ There is an $e\in X\setminus B_{\rho}$ such that $I(e)\leq0.$

Then $c=\underset{\gamma\in\Gamma}{\inf}\underset{0\leq t\leq1}{\max}%
I(\gamma\left(  t\right)  )\geq\alpha$ is a critical value of $I$ where
\[
\Gamma=\left\{  \gamma\in C^{0}\left(  \left[  0,1\right]  ,X\right)
,~\gamma(0)=0,~\gamma\left(  1\right)  =e\right\}  .
\]

\end{lemma}

\begin{lemma}
Let $\left(  X,\left\Vert \cdot\right\Vert _{X}\right)  $ be a real Banach
space and $I\in C^{1}\left(  X,%
\mathbb{R}
\right)  $ satisfies $I(0)=0$ and

$(i)$ There are constants $\rho,~\alpha>0$ such that $I|_{\partial B_{\rho}%
}\geq\alpha$.

$(ii)$ There is an $e\in X\setminus B_{\rho}$ such that $I(e)\leq0.$

Let $C_{M}$ be characterized by
\[
C_{M}=\underset{\gamma\in\Gamma}{\inf}\underset{0\leq t\leq1}{\max}%
I(\gamma\left(  t\right)  )
\]
where
\[
\Gamma=\left\{  \gamma\in C^{0}\left(  \left[  0,1\right]  ,X\right)
,~\gamma(0)=0,~\gamma\left(  1\right)  =e\right\}  .
\]
Then $I$ possesses a $(C)_{C_{M}}$ sequence.
\end{lemma}

As we remarked earlier, our results are motivated by the so-called
Moser-Trudinger inequality which can be found in \cite{M}. As we know, if
$\Omega\subset%
\mathbb{R}
^{N}$ $(N>p)$ is a bounded domain, then the Sobolev imbedding theorem states
that $W_{0}^{1,p}\left(  \Omega\right)  \subset L^{q}\left(  \Omega\right)  $,
for $1\leq q\leq p^{\ast}=\frac{pN}{N-p}$, or equivalently,
\[
\underset{u\in W_{0}^{1,p}\left(  \Omega\right)  ,~\left\Vert u\right\Vert
\leq1}{\sup}\int_{U}\left\vert u\right\vert ^{q}dx\leq C\left(  \Omega\right)
,~\text{for }1\leq q\leq p^{\ast},
\]
while the supremum is infinite for $q>p^{\ast}$. In the case $p=N$, it was
shown by Trudinger \cite{T} and Moser \cite{M} that the maximal growth is of
exponential type. More precisely, we have the following lemma:

\begin{lemma}
Let $u\in W_{0}^{1,N}\left(  \Omega\right)  $, then $\exp(\left\vert
u\right\vert ^{N/(N-1)})\in L^{q}\left(  \Omega\right)  $ for all $1\leq
q<\infty$. Moreover,
\[
\underset{u\in W_{0}^{1,N}\left(  \Omega\right)  ,~\left\Vert u\right\Vert
\leq1}{\sup}\int_{\Omega}\exp(\alpha\left\vert u\right\vert ^{N/(N-1)})dx\leq
C\left(  \Omega\right)  \text{ for }\alpha\leq\alpha_{N}.
\]
The inequality is optimal: for any growth $\exp(\alpha\left\vert u\right\vert
^{N/(N-1)})$ with $\alpha>\alpha_{N}$ the corresponding supremum is $+\infty$.
\end{lemma}

\section{\bigskip The improved subcritical polynomial growth $(SCPI)$-Proof of
Theorem 1}

In this section, we study the problem (\ref{1.2}) in the case $1<p<N$. As we
mentioned earlier, there have been a lot of papers about the existence of
nontrivial nonnegative solutions without the the $(AR)$-condition in the case
of subcritical polynomial growth. Nevertheless, almost all of them consider
the problem (\ref{1.2}) under the nonlinear term $f$ satisfies the condition
$(SCP)$ which is stronger than our condition $(SCPI)$. In \cite{LW}, the
authors had a similar result to ours by using the Nehari condition type to
replace for the $(AR)$ condition. Here, we will show that we can use a
suitable Mountain Pass Theorem to get our desired result.

\begin{lemma}
Let $f$ satisfy $(L1),~(L2),~(L4),~(SCPI)$. Then $J$ satisfies the conditions
$(i)$ and $(ii)$ of Lemma 1.
\end{lemma}

\begin{proof}
Let $u$ $\in W_{0}^{1,p}\left(  \Omega\right)  \setminus\left\{  0\right\}
,~u\geq0$. By $(L2),$ for all $M,$ there exists $d$ such that for all $\left(
x,s\right)  \in\Omega\times%
\mathbb{R}
^{+}$
\begin{equation}
F\left(  x,s\right)  \geq Ms^{p}-d. \label{Condition 1}%
\end{equation}
Then
\begin{align*}
J(tu)  &  \leq\frac{t^{p}}{p}\left\Vert u\right\Vert ^{p}-Mt^{p}\int_{\Omega
}\left\vert u\right\vert ^{p}dx+O(1)\\
&  =t^{p}\left(  \frac{\left\Vert u\right\Vert ^{p}}{p}-M\int_{\Omega
}\left\vert u\right\vert ^{p}dx\right)  +O(1)
\end{align*}
Now, choose $M>\frac{\left\Vert u\right\Vert ^{p}}{p\left\Vert u\right\Vert
_{p}^{p}}$, we have $J(tu)\rightarrow-\infty$ as $t\rightarrow\infty$, so $J$
satisfies $(ii)$ of Lemma 1.

Next, by $\left(  L4\right)  $ and $\left(  SCPI\right)  ,$ there exist
$C,\tau>0$ such that
\begin{equation}
F(x,s)\leq\frac{1}{p}\left(  \lambda_{1}-\tau\right)  \left\vert s\right\vert
^{p}+C\left\vert s\right\vert ^{p^{\ast}},~\forall\left(  x,s\right)
\in\Omega\times%
\mathbb{R}
\label{Condition 2}%
\end{equation}
Thus by the definition of $\lambda_{1}\left(  \Omega\right)  ~$and the Sobolev embedding:%

\[
J(u)\geq\frac{1}{p}\left(  1-\frac{\left(  \lambda_{1}-\tau\right)  }%
{\lambda_{1}}\right)  \left\Vert u\right\Vert ^{p}-C\left\Vert u\right\Vert
^{p^{\ast}}%
\]
Since $\tau>0$ and $p^{\ast}>p$, we may choose $\rho,\delta>0$ such that
$J(u)\geq\delta$ if $\left\Vert u\right\Vert =\rho$ and so, $J$ satisfies
$(i)$ of the Lemma 1.
\end{proof}

Next, we will check that $J$ satisfies the $(C)_{c}$ for all real numbers $c$.

\begin{lemma}
\bigskip Assume $(L1),~(L2),$ $(L3)$ and $(L4)$ hold. If $f$ has the improved
subcritical polynomial growth on $\Omega~(SCPI)$, then $J$ satisfies $(C)_{c}$
for all $c\in%
\mathbb{R}
.$
\end{lemma}

\begin{proof}
Let $\left\{  u_{n}\right\}  $ be a Cerami sequence in $W_{0}^{1,p}\left(
\Omega\right)  $ such that
\begin{align*}
\left(  1+\left\Vert u_{n}\right\Vert \right)  \left\Vert DJ(u_{n}%
)\right\Vert  &  \rightarrow0\\
J(u_{n})  &  \rightarrow c
\end{align*}
i.e.
\begin{align}
\left(  1+\left\Vert u_{n}\right\Vert \right)  \left\vert \int_{\Omega
}\left\vert \nabla u_{n}\right\vert ^{p-2}\nabla u_{n}\nabla vdx-\int_{\Omega
}f(x,u_{n})vdx\right\vert  &  \leq\varepsilon_{n}\left\Vert v\right\Vert
\label{3.1}\\
\frac{1}{p}\left\Vert u_{n}\right\Vert ^{p}-\int_{\Omega}F(x,u_{n})dx  &
\rightarrow c\nonumber
\end{align}
where $\varepsilon_{n}\overset{n\rightarrow\infty}{\rightarrow}0$. We first
show that $\left\{  u_{n}\right\}  $ is bounded which is our main purpose in
this paper. Indeed, suppose that
\begin{equation}
\left\Vert u_{n}\right\Vert \rightarrow\infty\label{3.2}%
\end{equation}
Setting
\[
v_{n}=\frac{u_{n}}{\left\Vert u_{n}\right\Vert }%
\]
then $\left\Vert v_{n}\right\Vert =1$ so we can suppose that $v_{n}%
\rightharpoonup v$ in $W_{0}^{1,p}\left(  \Omega\right)  $. We may similarly
show that $v_{n}^{+}\rightharpoonup v^{+}$ in $W_{0}^{1,p}\left(
\Omega\right)  $, where $w^{+}=\max\left\{  w,0\right\}  .$ Since $\Omega$ is
bounded, Sobolev's imbedding theorem implies that $\left\{
\begin{array}
[c]{l}%
v_{n}^{+}(x)\rightarrow v^{+}(x)\text{ a.e. in }\Omega\\
v_{n}^{+}\rightarrow v^{+}\text{ in }L^{q}\left(  \Omega\right)
,~\forall1\leq q<p^{\ast}%
\end{array}
\right.  .$ We wish to show that $v^{+}=0$ a.e. $\Omega.$ Indeed, if
$\Omega^{+}=\left\{  x\in\Omega:v^{+}\left(  x\right)  >0\right\}  $ has a
positive measure$,$ then in $\Omega^{+}$, we have
\[
\underset{n\rightarrow\infty}{\lim}u_{n}^{+}(x)=\underset{n\rightarrow\infty
}{\lim}v_{n}^{+}(x)\left\Vert u_{n}\right\Vert =+\infty
\]
and thus by $(L2):$
\[
\underset{n\rightarrow\infty}{\lim}\frac{F\left(  x,u_{n}^{+}(x)\right)
}{\left\vert u_{n}^{+}(x)\right\vert ^{p}}=+\infty\text{ a.e. in }\Omega^{+}%
\]
This means that
\begin{equation}
\underset{n\rightarrow\infty}{\lim}\frac{F\left(  x,u_{n}^{+}(x)\right)
}{\left\vert u_{n}^{+}(x)\right\vert ^{p}}\left\vert v_{n}^{+}(x)\right\vert
^{p}=+\infty\text{ a.e. in }\Omega^{+} \label{3.3}%
\end{equation}
and so
\begin{equation}
\int_{\Omega^{+}}\underset{n\rightarrow\infty}{\lim\inf}\frac{F\left(
x,u_{n}^{+}(x)\right)  }{\left\vert u_{n}^{+}(x)\right\vert ^{p}}\left\vert
v_{n}^{+}(x)\right\vert ^{p}dx=+\infty\label{3.4}%
\end{equation}
Also, by ($\ref{3.1}$), we see that
\[
\left\Vert u_{n}\right\Vert ^{p}=pc+p\int_{\Omega}F(x,u_{n}^{+}\left(
x\right)  )dx+o(1)
\]
which implies that
\[
\int_{\Omega}F(x,u_{n}^{+}\left(  x\right)  )dx\rightarrow+\infty
\]
and
\begin{align}
&  \underset{n\rightarrow\infty}{\lim\inf}\int_{\Omega}\frac{F\left(
x,u_{n}^{+}(x)\right)  }{\left\Vert u_{n}\right\Vert ^{p}}dx\label{3.5}\\
&  =\underset{n\rightarrow\infty}{\lim\inf}\frac{\int_{\Omega^{+}}F\left(
x,u_{n}^{+}(x)\right)  dx}{pc+p\int_{\Omega}F(x,u_{n}^{+}\left(  x\right)
)dx+o(1)}\nonumber\\
&  =\frac{1}{p}\nonumber
\end{align}
Now, note that $F(x,s)\geq0$, by Fatou's lemma and (\ref{3.4}) and
(\ref{3.5}), we get a contradiction. So $v\leq0$ a.e.

Letting $t_{n}\in\left[  0,1\right]  $ such that
\[
J\left(  t_{n}u_{n}\right)  =\underset{t\in\left[  0,1\right]  }{\max}J\left(
tu_{n}\right)
\]
For all $R>0$, by $(SCPI)$, there exists $C>0$ such that
\begin{equation}
F(x,s)\leq C\left\vert s\right\vert +\frac{1}{R^{p^{\ast}}}s^{p^{\ast}%
},~\forall\left(  x,s\right)  \in\Omega\times%
\mathbb{R}
. \label{3.6}%
\end{equation}
Also since $\left\Vert u_{n}\right\Vert \rightarrow\infty$, we have for $n$
sufficient large:
\begin{equation}
J\left(  t_{n}u_{n}\right)  \geq J\left(  \frac{R}{\left\Vert u_{n}\right\Vert
}u_{n}\right)  =J\left(  Rv_{n}\right)  \label{3.7}%
\end{equation}
and by (\ref{3.6}) with note that $\int_{\Omega}F\left(  x,v_{n}\right)
dx=\int_{\Omega}F\left(  x,v_{n}^{+}\right)  dx$:
\begin{align}
pJ\left(  Rv_{n}\right)   &  \geq R^{p}-pC\int_{\Omega}\left\vert Rv_{n}%
^{+}(x)\right\vert dx-\frac{p}{R^{p^{\ast}}}\int_{\Omega}\left\vert Rv_{n}%
^{+}\right\vert ^{p^{\ast}}dx\label{3.8}\\
&  =R^{p}-pRC\int_{\Omega}\left\vert v_{n}^{+}(x)\right\vert dx-p\int_{\Omega
}\left\vert v_{n}^{+}\right\vert ^{p^{\ast}}dx\nonumber
\end{align}
Since $v_{n}^{+}\rightharpoonup0$ weakly in $W_{0}^{1,p}\left(  \Omega\right)
$, thus $\int_{\Omega}\left\vert v_{n}^{+}\right\vert ^{p^{\ast}}dx$ is
bounded by a universal constant $C\left(  \Omega\right)  >0$ and also
$\int_{\Omega}\left\vert v_{n}^{+}(x)\right\vert dx\rightarrow0$. Thus if we
let $n\rightarrow\infty$ in (\ref{3.8}), and then let $R\rightarrow\infty$ and
using (\ref{3.7}), we get
\begin{equation}
J\left(  t_{n}u_{n}\right)  \rightarrow\infty\label{3.9}%
\end{equation}
Note that $J(0)=0$ and $J(u_{n})\rightarrow c$, we can suppose that $t_{n}%
\in\left(  0,1\right)  $. Thus $DJ(t_{n}u_{n})t_{n}u_{n}=0,$ i.e.,%
\[
t_{n}^{p}\left\Vert u_{n}\right\Vert ^{p}=\int_{\Omega}f\left(  x,t_{n}%
u_{n}\right)  t_{n}u_{n}dx
\]
Also, by (\ref{3.1})%
\begin{align*}
\int_{\Omega}\left[  f\left(  x,u_{n}\right)  u_{n}-pF\left(  x,u_{n}\right)
\right]  dx  &  =\left\Vert u_{n}\right\Vert ^{p}+pc-\left\Vert u_{n}%
\right\Vert ^{p}+o(1)\\
&  =pc+o(1)
\end{align*}
So by $(L3):$
\begin{align*}
pJ\left(  t_{n}u_{n}\right)   &  =t_{n}^{p}\left\Vert u_{n}\right\Vert
^{p}-p\int_{\Omega}F\left(  x,t_{n}u_{n}\right)  dx\\
&  =\int_{\Omega}\left[  f\left(  x,t_{n}u_{n}\right)  t_{n}u_{n}-pF\left(
x,t_{n}u_{n}\right)  \right]  dx\\
&  \leq\theta\int_{\Omega}\left[  f\left(  x,u_{n}\right)  u_{n}-pF\left(
x,u_{n}\right)  \right]  dx+O(1)\\
&  \leq O(1)
\end{align*}
which is a contraction to (\ref{3.9}). This proves that $\left\{
u_{n}\right\}  $ is bounded in $W_{0}^{1,p}\left(  \Omega\right)  $. Without
loss of generality, we can suppose that
\[
\left\{
\begin{array}
[c]{l}%
u_{n}\rightharpoonup u\text{ in }W_{0}^{1,p}\left(  \Omega\right) \\
u_{n}\left(  x\right)  \rightarrow u\left(  x\right)  \text{ a.e. }\Omega\\
u_{n}\longrightarrow u\text{ in }L^{q}\left(  \Omega\right)  ,~\forall1\leq
q<p^{\ast}.
\end{array}
\right.
\]
Now, since $f$ has the subcritical growth on $\Omega$, for every
$\varepsilon>0$, we can find a constant $C(\varepsilon)>0$ such that
\[
f\left(  x,s\right)  \leq C(\varepsilon)+\varepsilon\left\vert s\right\vert
^{p^{\ast}-1},~\forall\left(  x,s\right)  \in\Omega\times%
\mathbb{R}
\]
then%
\begin{align*}
&  \left\vert \int_{\Omega}f\left(  x,u_{n}\right)  \left(  u_{n}-u\right)
dx\right\vert \\
\leq &  C(\varepsilon)\int_{\Omega}\left\vert \left(  u_{n}-u\right)
\right\vert dx+\varepsilon\int_{\Omega}\left\vert \left(  u_{n}-u\right)
\right\vert \left\vert u_{n}\right\vert ^{p^{\ast}-1}dx\\
\leq &  C(\varepsilon)\int_{\Omega}\left\vert \left(  u_{n}-u\right)
\right\vert dx+\varepsilon\left(  \int_{\Omega}\left(  \left\vert
u_{n}\right\vert ^{p^{\ast}-1}\right)  ^{p^{\ast}/(p^{\ast}-1)}dx\right)
^{(p^{\ast}-1)/p^{\ast}}\left(  \int_{\Omega}\left\vert u_{n}-u\right\vert
^{p^{\ast}}dx\right)  ^{1/p^{\ast}}\\
\leq &  C(\varepsilon)\int_{\Omega}\left\vert \left(  u_{n}-u\right)
\right\vert dx+\varepsilon C\left(  \Omega\right)
\end{align*}
Similarly, since $u_{n}\rightharpoonup u$ in $W_{0}^{1,p}\left(
\Omega\right)  ,~\int_{\Omega}\left\vert \left(  u_{n}-u\right)  \right\vert
dx\rightarrow0$. Since $\varepsilon>0$ is arbitrary, we can conclude that
$\int_{\Omega}f\left(  x,u_{n}\right)  \left(  u_{n}-u\right)  dx\rightarrow
0$. Thus we can conclude that
\begin{equation}
\int_{\Omega}\left(  f\left(  x,u_{n}\right)  -f\left(  x,u\right)  \right)
\left(  u_{n}-u\right)  dx\overset{n\rightarrow\infty}{\rightarrow}0
\label{3.10}%
\end{equation}
By (\ref{3.1}), we have
\begin{equation}
\left\langle DJ(u_{n})-DJ(u),(u_{n}-u)\right\rangle \overset{n\rightarrow
\infty}{\rightarrow}0 \label{3.11}%
\end{equation}
From (\ref{3.10}) and (\ref{3.11}), we get
\[
\int_{\Omega}\left(  \left\vert \nabla u_{n}\right\vert ^{p-2}\nabla
u_{n}-\left\vert \nabla u\right\vert ^{p-2}\nabla u\right)  \left(  \nabla
u_{n}-\nabla u\right)  \rightarrow0
\]
Using an elementary inequality
\[
2^{2-p}\left\vert b-a\right\vert ^{p}\leq\left\langle \left\vert b\right\vert
^{p-2}b-\left\vert a\right\vert ^{p-2}a,b-a\right\rangle ,~\forall a,b\in%
\mathbb{R}
^{p}%
\]
we can deduce that
\[
\nabla u_{n}\rightarrow\nabla u\text{ in }L^{p}\left(  \Omega\right)
\]
So we have $u_{n}\overset{n\rightarrow\infty}{\rightarrow}u$ strongly in
$W_{0}^{1,p}\left(  \Omega\right)  $ which means that $J$ satisfies $(C)_{c}$.
\end{proof}

\subsection{Proof of Theorem 1}

Combing Lemma 5 and Mountain Pass Theorem (Lemma 1), we can easily deduce that the problem
(\ref{1.2}) has a nontrivial weak solution.

\section{The subcritical exponential growth-Proof of Theorem 2}

In this section, we will study the problem (\ref{1.2}) in the case $p=N\geq3$
and $f$ satisfies the $(SCE)$. As far as we know, this appears to be the first
work with the $(AR)$-condition free in the subcritical exponential growth.

\subsection{The geometry of the functional J}

In this subsection, we will check the Mountain Pass properties of the
functional $J$. Similar to Lemma 4, we have the following lemma:

\begin{lemma}
Let $f$ satisfy $(L2)$. Then $J(tu)\rightarrow-\infty$ as $t\rightarrow\infty$
for all nonnegative function $u$ $\in W_{0}^{1,N}\left(  \Omega\right)
\setminus\left\{  0\right\}  .$
\end{lemma}

This means that the condition $(i)$ in Lemma 1 is satisfied. Now, we will
check the second one:

\begin{lemma}
Let $f$ satisfy $(L1),~(L4),~(SCE).$ Then there exist $\delta,\rho>0$ such
that
\[
J(u)\geq\delta\text{ if }\left\Vert u\right\Vert =\rho
\]

\end{lemma}

\begin{proof}
By $\left(  L4\right)  $ and $\left(  SCE\right)  ,$ there exist $\kappa
,\tau>0~$and $q>N~$such that
\[
F(x,s)\leq\frac{1}{N}\left(  \lambda_{1}-\tau\right)  \left\vert s\right\vert
^{N}+C\exp\left(  \kappa\left\vert s\right\vert ^{N/(N-1)}\right)  \left\vert
s\right\vert ^{q},~\forall\left(  x,s\right)  \in\Omega\times%
\mathbb{R}%
\]
By Holder's inequality and the Moser-Trudinger embedding, we have:
\begin{align*}
\int_{\Omega}\exp\left(  \kappa\left\vert u\right\vert ^{N/(N-1)}\right)
\left\vert u\right\vert ^{q}dx  &  \leq\left(  \int_{\Omega}\exp\left(  \kappa
r\left\Vert u\right\Vert ^{N/(N-1)}\left(  \frac{\left\vert u\right\vert
}{\left\Vert u\right\Vert }\right)  ^{N/(N-1)}\right)  dx\right)
^{1/r}\left(  \int_{\Omega}\left\vert u\right\vert ^{r^{\prime}q}dx\right)
^{1/r^{\prime}}\\
&  \leq C\left(  \int_{\Omega}\left\vert u\right\vert ^{r^{\prime}q}dx\right)
^{1/r^{\prime}}%
\end{align*}
if $r>1$ sufficiently close to 1 and $\left\Vert u\right\Vert \leq\sigma$,
where $\kappa r\sigma^{N/(N-1)}<\alpha_{N}$. Thus by the definition of
$\lambda_{1}~$and the Sobolev embedding:%

\[
J(u)\geq\frac{1}{N}\left(  1-\frac{\left(  \lambda_{1}-\tau\right)  }%
{\lambda_{1}}\right)  \left\Vert u\right\Vert ^{N}-C\left\Vert u\right\Vert
^{q}%
\]
Since $\tau>0$ and $q>N$, we may choose $\rho,\delta>0$ such that
$J(u)\geq\delta$ if $\left\Vert u\right\Vert =\rho$.
\end{proof}

Again, it's very important to check that $J$ satisfies the $(C)_{c}$ for all
real numbers $c$. Similar to what we have shown in the previous section, we
have the following lemma:

\begin{lemma}
\bigskip Assume $(L1),~(L2),$ $(L3)$ and $(L4)$ hold. If $f$ has subcritical
exponential growth on $\Omega~(SCE)$, then $J$ satisfies $(C)_{c}$ for all
$c\in%
\mathbb{R}
.$
\end{lemma}

\begin{proof}
Let $\left\{  u_{n}\right\}  $ be a Cerami sequence in $W_{0}^{1,N}\left(
\Omega\right)  $ such that
\begin{align*}
\left(  1+\left\Vert u_{n}\right\Vert \right)  \left\Vert DJ(u_{n}%
)\right\Vert  &  \rightarrow0\\
J(u_{n})  &  \rightarrow c
\end{align*}
i.e.
\begin{align}
\left(  1+\left\Vert u_{n}\right\Vert \right)  \left\vert \int_{\Omega
}\left\vert \nabla u_{n}\right\vert ^{N-2}\nabla u_{n}\nabla vdx-\int_{\Omega
}f(x,u_{n})vdx\right\vert  &  \leq\varepsilon_{n}\left\Vert v\right\Vert
\label{4.1}\\
\frac{1}{N}\left\Vert u_{n}\right\Vert ^{N}-\int_{\Omega}F(x,u_{n})dx  &
\rightarrow c\nonumber
\end{align}
where $\varepsilon_{n}\overset{n\rightarrow\infty}{\rightarrow}0$. We will
show that $\left\{  u_{n}\right\}  $ is bounded. Again, suppose that
\begin{equation}
\left\Vert u_{n}\right\Vert \rightarrow\infty\label{4.2}%
\end{equation}
Setting
\[
v_{n}=\frac{u_{n}}{\left\Vert u_{n}\right\Vert }%
\]
then $\left\Vert v_{n}\right\Vert =1$. We can then suppose that $v_{n}%
\rightharpoonup v$ in $W_{0}^{1,N}\left(  \Omega\right)  $ (up to a
subsequence) . We may similarly show that $v_{n}^{+}\rightharpoonup0$ in
$W_{0}^{1,N}\left(  \Omega\right)  $, where $w^{+}=\max\left\{  w,0\right\}
.$

Again, let $t_{n}\in\left[  0,1\right]  $ such that
\[
J\left(  t_{n}u_{n}\right)  =\underset{t\in\left[  0,1\right]  }{\max}J\left(
tu_{n}\right)
\]
For any given $R>0$, by $(SCE)$, there exists $C=C(R)>0$ such that
\begin{equation}
F(x,s)\leq C\left\vert s\right\vert +\exp\left(  \frac{\alpha_{N}}%
{R^{N/(N-1)}}s^{N/(N-1)}\right)  ,~\forall\left(  x,s\right)  \in\Omega\times%
\mathbb{R}
. \label{4.6}%
\end{equation}
Also since $\left\Vert u_{n}\right\Vert \rightarrow\infty$, we have
\begin{equation}
J\left(  t_{n}u_{n}\right)  \geq J\left(  \frac{R}{\left\Vert u_{n}\right\Vert
}u_{n}\right)  =J\left(  Rv_{n}\right)  \label{4.7}%
\end{equation}
and by (\ref{4.6}), $\left\Vert v_{n}\right\Vert =1$ and the fact that
$\int_{\Omega}F\left(  x,v_{n}\right)  dx=\int_{\Omega}F\left(  x,v_{n}%
^{+}\right)  dx$, we get
\begin{align}
NJ\left(  Rv_{n}\right)   &  \geq R^{N}-NCR\int_{\Omega}\left\vert v_{n}%
^{+}(x)\right\vert dx-N\int_{\Omega}\exp\left(  \alpha_{N}\left\vert v_{n}%
^{+}(x)\right\vert ^{N/(N-1)}\right)  dx\label{4.8}\\
&  \geq R^{N}-NCR\int_{\Omega}\left\vert v_{n}^{+}(x)\right\vert
dx-N\int_{\Omega}\exp\left(  \alpha_{N}\left\vert v_{n}(x)\right\vert
^{N/(N-1)}\right)  dx\nonumber
\end{align}
Since $\left\Vert v_{n}\right\Vert =1,$ we have that $\int_{\Omega}\exp\left(
\alpha_{N}\left\vert v_{n}(x)\right\vert ^{N/(N-1)}\right)  dx$ is bounded by
a universal constant $C\left(  \Omega\right)  >0$ by the Moser-Trudinger
inequality (Lemma 3). Also, since $v_{n}^{+}\rightharpoonup0$ in $W_{0}%
^{1,N}\left(  \Omega\right)  $, we have that $\int_{\Omega}\left\vert
v_{n}^{+}(x)\right\vert dx\rightarrow0$. Thus using (\ref{4.7}) and letting
$n\rightarrow\infty$ in (\ref{4.8}), and then letting $R\rightarrow\infty$, we
get
\begin{equation}
J\left(  t_{n}u_{n}\right)  \rightarrow\infty\label{4.9}%
\end{equation}
Note that $J(0)=0$ and $J(u_{n})\rightarrow c$, we can suppose that $t_{n}%
\in\left(  0,1\right)  $. Thus since $DJ(t_{n}u_{n})t_{n}u_{n}=0,$%
\[
t_{n}^{N}\left\Vert u_{n}\right\Vert ^{N}=\int_{\Omega}f\left(  x,t_{n}%
u_{n}\right)  t_{n}u_{n}dx
\]
So by $(L3):$
\begin{align*}
NJ\left(  t_{n}u_{n}\right)   &  =t_{n}^{N}\left\Vert u_{n}\right\Vert
^{N}-N\int_{\Omega}F\left(  x,t_{n}u_{n}\right)  dx\\
&  =\int_{\Omega}\left[  f\left(  x,t_{n}u_{n}\right)  t_{n}u_{n}-NF\left(
x,t_{n}u_{n}\right)  \right]  dx\\
&  \leq\theta\int_{\Omega}\left[  f\left(  x,u_{n}\right)  u_{n}-NF\left(
x,u_{n}\right)  \right]  dx+O(1)
\end{align*}
Also, by (\ref{3.1}), we have%
\begin{align*}
\int_{\Omega}\left[  f\left(  x,u_{n}\right)  u_{n}-NF\left(  x,u_{n}\right)
\right]  dx  &  =\left\Vert u_{n}\right\Vert ^{N}+Nc-\left\Vert u_{n}%
\right\Vert ^{N}+o(1)\\
&  =Nc+o(1)
\end{align*}
which is a contraction to (\ref{3.9}). This proves that $\left\{
u_{n}\right\}  $ is bounded in $W_{0}^{1,N}\left(  \Omega\right)  $. Without
loss of generality, suppose that
\[
\left\{
\begin{array}
[c]{l}%
\left\Vert u_{n}\right\Vert \leq K\\
u_{n}\rightharpoonup u\text{ in }W_{0}^{1,N}\left(  \Omega\right) \\
u_{n}\left(  x\right)  \rightarrow u\left(  x\right)  \text{ a.e. }\Omega\\
u_{n}\longrightarrow u\text{ in }L^{p}\left(  \Omega\right)  ,~\forall p\geq1.
\end{array}
\right.
\]
Now, since $f$ has the subcritical exponential growth (SCE) on $\Omega$, we
can find a constant $c_{K}>0$ such that
\[
f\left(  x,s\right)  \leq c_{K}\exp\left(  \frac{\alpha_{N}}{2K^{N/(N-1)}%
}\left\vert s\right\vert ^{N/(N-1)}\right)  ,~\forall\left(  x,s\right)
\in\Omega\times%
\mathbb{R}
\]
then by the Moser-Trudinger inequality,
\begin{align*}
\left\vert \int_{\Omega}f\left(  x,u_{n}\right)  \left(  u_{n}-u\right)
dx\right\vert  &  \leq\int_{\Omega}\left\vert f\left(  x,u_{n}\right)  \left(
u_{n}-u\right)  \right\vert dx\\
&  \leq\left(  \int_{\Omega}\left\vert f\left(  x,u_{n}\right)  \right\vert
^{2}dx\right)  ^{1/2}\left(  \int_{\Omega}\left\vert u_{n}-u\right\vert
^{2}dx\right)  ^{1/2}\\
&  \leq C\left(  \int_{\Omega}\exp\left(  \frac{\alpha_{N}}{K^{N/(N-1)}%
}\left\vert u_{n}\right\vert ^{N/(N-1)}\right)  dx\right)  ^{1/2}\left\Vert
u_{n}-u\right\Vert _{2}\\
&  \leq C\left(  \int_{\Omega}\exp\left(  \frac{\alpha_{N}}{K^{N/(N-1)}%
}\left\Vert u_{n}\right\Vert ^{N/(N-1)}\left\vert \frac{u_{n}}{\left\Vert
u_{n}\right\Vert }\right\vert ^{N/(N-1)}\right)  dx\right)  ^{1/2}\left\Vert
u_{n}-u\right\Vert _{2}\\
&  \leq C\left\Vert u_{n}-u\right\Vert _{2}\overset{n\rightarrow\infty
}{\rightarrow}0.
\end{align*}
Similarly, since $u_{n}\rightharpoonup u$ in $W_{0}^{1,N}\left(
\Omega\right)  ,~\int_{\Omega}f\left(  x,u\right)  \left(  u_{n}-u\right)
dx\rightarrow0$. Thus we can conclude that
\begin{equation}
\int_{\Omega}\left(  f\left(  x,u_{n}\right)  -f\left(  x,u\right)  \right)
\left(  u_{n}-u\right)  dx\overset{n\rightarrow\infty}{\rightarrow}0
\label{4.10}%
\end{equation}
Also, by (\ref{4.1}) we have
\begin{equation}
\left\langle DJ(u_{n})-DJ(u),(u_{n}-u)\right\rangle \overset{n\rightarrow
\infty}{\rightarrow}0 \label{4.11}%
\end{equation}
From (\ref{3.10}) and (\ref{3.11}), we get
\[
\int_{\Omega}\left(  \left\vert \nabla u_{n}\right\vert ^{N-2}\nabla
u_{n}-\left\vert \nabla u\right\vert ^{N-2}\nabla u\right)  \left(  \nabla
u_{n}-\nabla u\right)  \rightarrow0
\]
Using an elementary inequality
\[
2^{2-N}\left\vert b-a\right\vert ^{N}\leq\left\langle \left\vert b\right\vert
^{N-2}b-\left\vert a\right\vert ^{N-2}a,b-a\right\rangle ,~\forall a,b\in%
\mathbb{R}
^{N}%
\]
we can deduce that
\[
\nabla u_{n}\rightarrow\nabla u\text{ in }L^{N}\left(  \Omega\right)
\]
So we have $u_{n}\overset{n\rightarrow\infty}{\rightarrow}u$ strongly in
$W_{0}^{1,N}\left(  \Omega\right)  $ which shows that $J$ satisfies $(C)_{c}$.
\end{proof}

\subsection{Proof of Theorem 2}

Again, by Lemma 8 and Mountain Pass Theorem (Lemma 1), we can easily deduce that the
problem (\ref{1.2}) has a nontrivial weak solution.

\section{The critical exponential growth-Proof of Theorem 3}

In this section, we study the problem (\ref{1.2}) where $\Omega$ is the
bounded domain in $%
\mathbb{R}
^{N}$ and $f$ has the critical growth $(CR),$ say, at $\alpha_{0}>0$. Recall
that then we have
\[
\underset{u\rightarrow+\infty~}{\lim}\frac{\left\vert f\left(  x,u\right)
\right\vert }{\exp\left(  \alpha\left\vert u\right\vert ^{N/(N-1)}\right)
}=0,\text{ uniformly on }x\in\Omega,~\forall\alpha>\alpha_{0}%
\]
and
\[
\underset{u\rightarrow+\infty~}{\lim}\frac{\left\vert f\left(  x,u\right)
\right\vert }{\exp\left(  \alpha\left\vert u\right\vert ^{N/(N-1)}\right)
}=+\infty,\text{ uniformly on }x\in\Omega,~\forall\alpha<\alpha_{0}%
\]
We now start the proof of Theorem 3.

\begin{proof}
Similar to the previous two sections, by our conditions, we see that our
Euler-Lagrange function associated to the problem (\ref{1.2}) has the
Palais-Smale geometry properties. Now we consider the Moser functions:%
\[
\widetilde{M}_{n}(x)=\omega_{N-1}^{-1/N}\left\{
\begin{array}
[c]{l}%
\left(  \log n\right)  ^{(N-1)/N},~0\leq\left\vert x\right\vert \leq1/n\\
\frac{\log(1/\left\vert x\right\vert )}{\left(  \log n\right)  ^{1/N}%
},~1/n\leq\left\vert x\right\vert \leq1\\
0,~~~~\ \ \ \ \ ~\ 1\leq\left\vert x\right\vert
\end{array}
\right.
\]
We see that $\widetilde{M}_{n}\in W_{0}^{1,N}\left(  B_{1}(0)\right)  $ and
$\left\Vert \widetilde{M}_{n}\right\Vert =1,~\forall n\in%
\mathbb{N}
.$ Since $d$ is the inner radius of $\Omega$, we can find $x_{0}\in\Omega$
such that $B_{d}(x_{0})\subset\Omega$. Letting $M_{n}(x)=\widetilde{M}%
_{n}(\frac{x-x_{0}}{d})$, which are in $W_{0}^{1,N}\left(  \Omega\right)
,\left\Vert M_{n}\right\Vert =1$ and $suppM_{n}=B_{d}(x_{0})$. As in the proof
of Theorem 1.3 in \cite{FMR}, we can conclude that
\[
\max\left\{  J(tM_{n}):t\geq0\right\}  <\frac{1}{N}\left(  \frac{\alpha_{N}%
}{\alpha_{0}}\right)  ^{N-1}%
\]
It can be checked easily by a similar argument to that in the previous section
that $J$ satisfies the condition $(i)$ and $(ii)$ of Lemma 2 (See Lemmas 6 and
7). So, we can find a Cerami sequence $\left\{  u_{n}\right\}  $ in
$W_{0}^{1,N}\left(  \Omega\right)  $ such that
\begin{align}
\left(  1+\left\Vert u_{n}\right\Vert \right)  \left\Vert DJ(u_{n}%
)\right\Vert  &  \rightarrow0\label{5.0}\\
J(u_{n})  &  \rightarrow C_{M}<\frac{1}{N}\left(  \frac{\alpha_{N}}{\alpha
_{0}}\right)  ^{N-1}\nonumber
\end{align}
We again want to show that $\left\{  u_{n}\right\}  $ is bounded in
$W_{0}^{1,N}\left(  \Omega\right)  $. Indeed, if we suppose that $\left\{
u_{n}\right\}  $ is unbounded, then using the same argument to that used in
the previous two sections, we can get that
\[
v_{n}^{+}\rightharpoonup0\text{ in }W_{0}^{1,N}\left(  \Omega\right)  \text{
where }v_{n}=\frac{u_{n}}{\left\Vert u_{n}\right\Vert }.
\]
Let $t_{n}\in\left[  0,1\right]  $ such that
\[
J\left(  t_{n}u_{n}\right)  =\underset{t\in\left[  0,1\right]  }{\max}J\left(
tu_{n}\right)
\]
Let $R\in\left(  0,\left(  \frac{\alpha_{N}}{\alpha_{0}}\right)
^{(N-1)/N}\right)  $ and choose $\varepsilon=\frac{\alpha_{N}}{R^{N/(N-1)}}-$
$\alpha_{0}>0$, by condition $(CG)$, there exists $C>0$ such that
\begin{equation}
F(x,s)\leq C\left\vert s\right\vert +\left\vert \frac{\alpha_{N}}{R^{N/(N-1)}%
}-\alpha_{0}\right\vert \exp\left(  \left(  \alpha_{0}+\varepsilon\right)
s^{N/(N-1)}\right)  ,~\forall\left(  x,s\right)  \in\Omega\times%
\mathbb{R}
. \label{5.1}%
\end{equation}
Since $\left\Vert u_{n}\right\Vert \rightarrow\infty$, we have
\begin{equation}
J\left(  t_{n}u_{n}\right)  \geq J\left(  \frac{R}{\left\Vert u_{n}\right\Vert
}u_{n}\right)  =J\left(  Rv_{n}\right)  \label{5.2}%
\end{equation}
and by (\ref{5.1}) and noticing $\left\Vert v_{n}\right\Vert =1$, we have
\begin{equation}
NJ\left(  Rv_{n}\right)  \geq R^{N}-NCR\int_{\Omega}\left\vert v_{n}%
^{+}(x)\right\vert dx-N\left\vert \frac{\alpha_{N}}{R^{N/(N-1)}}-\alpha
_{0}\right\vert \int_{\Omega}\exp\left(  \left(  \alpha_{0}+\varepsilon
\right)  R^{N/(N-1)}v_{n}^{N/(N-1)}(x)\right)  dx \label{5.3}%
\end{equation}
By the Moser-Trudinger inequality (Lemma 3),
\[
\int_{\Omega}\exp\left(  \left(  \alpha_{0}+\varepsilon\right)  R^{N/(N-1)}%
v_{n}^{N/(N-1)}(x)\right)  dx=\int_{\Omega}\exp\left(  \alpha_{N}%
v_{n}^{N/(N-1)}(x)\right)  dx
\]
is bounded by an universal constant $C\left(  \Omega\right)  >0$ thanks to the
choice of $\varepsilon$. Also, since $v_{n}^{+}\rightharpoonup0$ in
$W_{0}^{1,N}\left(  \Omega\right)  $, $\int_{\Omega}\left\vert v_{n}%
^{+}(x)\right\vert dx\rightarrow0$. Thus if we let $n\rightarrow\infty$ in
(\ref{5.3}), and then let $R\rightarrow\left[  \left(  \frac{\alpha_{N}%
}{\alpha_{0}}\right)  ^{(N-1)/N}\right]  ^{-}$ and using (\ref{5.2}), we get
\begin{equation}
\underset{n\rightarrow\infty}{\lim\inf}J\left(  t_{n}u_{n}\right)  \geq
\frac{1}{N}\left(  \frac{\alpha_{N}}{\alpha_{0}}\right)  ^{N-1}>C_{M}\text{.}
\label{5.4}%
\end{equation}
Note that $J(0)=0$ and $J(u_{n})\rightarrow C_{M}$, we can suppose that
$t_{n}\in\left(  0,1\right)  $. Thus since $DJ(t_{n}u_{n})t_{n}u_{n}=0,$%
\[
t_{n}^{N}\left\Vert u_{n}\right\Vert ^{N}=\int_{\Omega}f\left(  x,t_{n}%
u_{n}\right)  t_{n}u_{n}dx
\]
Also, by (\ref{5.0})%
\begin{align*}
\int_{\Omega}\left[  f\left(  x,u_{n}\right)  u_{n}-NF\left(  x,u_{n}\right)
\right]  dx  &  =\left\Vert u_{n}\right\Vert ^{N}+NC_{M}-\left\Vert
u_{n}\right\Vert ^{N}+o(1)\\
&  =NC_{M}+o(1)
\end{align*}
So by $(L3):$
\begin{align*}
NJ\left(  t_{n}u_{n}\right)   &  =t_{n}^{N}\left\Vert u_{n}\right\Vert
^{N}-N\int_{\Omega}F\left(  x,t_{n}u_{n}\right)  dx\\
&  =\int_{\Omega}\left[  f\left(  x,t_{n}u_{n}\right)  t_{n}u_{n}-NF\left(
x,t_{n}u_{n}\right)  \right]  dx\\
&  \leq\int_{\Omega}\left[  f\left(  x,u_{n}\right)  u_{n}-NF\left(
x,u_{n}\right)  \right]  dx\\
&  =NC_{M}+o(1)
\end{align*}
which is a contraction to (\ref{5.4}). This proves that $\left\{
u_{n}\right\}  $ is bounded in $W_{0}^{1,N}\left(  \Omega\right)  $. Now,
following the proof of Lemma 4 in \cite{O}, we can prove that $u$ is a weak
solution of (\ref{1.2}). So the last remaining point that we need to show is
the nontriviality of $u.$ However, we can get this thanks to our assumption
$(L6)$. Indeed, suppose $u=0$. Arguing as in \cite{O}, we get $f(x,u_{n}%
)\rightarrow0$ in $L^{1}\left(  \Omega\right)  $. Thanks to $(L6)$,
$F(x,u_{n})\rightarrow0$ in $L^{1}\left(  \Omega\right)  $ and we can get
\[
\underset{n\rightarrow\infty}{\lim}\left\Vert u_{n}\right\Vert ^{N}%
=NC_{M}<\left(  \frac{\alpha_{N}}{\alpha_{0}}\right)  ^{N-1}%
\]
and again, follows the proof in \bigskip\cite{O}, we have a contradiction. The
proof is now completed.
\end{proof}

\end{document}